\documentclass[a4paper,reqno]{amsart}

\usepackage{amsthm,amssymb,amsmath}
\usepackage{enumerate,enumitem}
\usepackage[font=small]{caption}
\usepackage{mathtools}
\usepackage{color}
\usepackage{graphicx}

\setlist{leftmargin=9mm}

\newtheorem{theorem}{Theorem}[section]

\newtheorem{lemma}[theorem]{Lemma}

\theoremstyle{definition}

\newtheorem*{conjecture}{Conjecture}

\theoremstyle{remark}
\newtheorem{remark}[theorem]{Remark}
\newtheorem*{remark*}{Remark}

\numberwithin{equation}{section}

\newcommand{\Oo}{\mathcal{O}}

\newcommand{\N}{\mathbb{N}}

\newcommand{\R}{\mathbb{R}}
\newcommand{\C}{\mathbb{C}}

\newcommand{\eps}{\varepsilon}

\newcommand{\dd}{\,{\rm d}}

\mathtoolsset{showonlyrefs}

\numberwithin{equation}{section}

\usepackage{color}
\usepackage[normalem]{ulem}
\definecolor{DarkBlue}{rgb}{0,0.1,0.7}  
\definecolor{DarkGreen}{rgb}{0,0.5,0.1}

\newcommand\soutD{\bgroup\markoverwith
{\textcolor{DarkGreen}{\rule[.5ex]{2pt}{1pt}}}\ULon}
\newcommand{\Hm}[1]{\leavevmode{\marginpar{\tiny%
$\hbox to 0mm{\hspace*{-0.5mm}$\leftarrow$\hss}%
\vcenter{\vrule depth 0.1mm height 0.1mm width \the\marginparwidth}%
\hbox to
0mm{\hss$\rightarrow$\hspace*{-0.5mm}}$\\\relax\raggedright #1}}}

\title{An improved discrete Rellich inequality on the half-line}

\author{Borbala Gerhat}
\address[Borbala Gerhat]{Department of Mathematics, Faculty of Nuclear Sciences and Physical Engineering, Czech Technical University in Prague, Trojanova 13, 120 00 Praha, Czech Republic}
\email{borbala.gerhat@fjfi.cvut.cz}

\author{David Krej\v ci\v r\'ik}
\address[David Krej\v ci\v r\'ik]{Department of Mathematics, Faculty of Nuclear Sciences and Physical Engineering, Czech Technical University in Prague, Trojanova 13, 120 00 Praha, Czech Republic}
\email{david.krejcirik@fjfi.cvut.cz}

\author{Franti\v sek \v Stampach}
\address[Franti\v sek \v Stampach]{Department of Mathematics, Faculty of Nuclear Sciences and Physical Engineering, Czech Technical University in Prague, Trojanova 13, 120 00 Praha, Czech Republic}
\email{frantisek.stampach@fjfi.cvut.cz}

\date{19 June 2022}

\begin{document}

\begin{abstract}
 Based on a new idea of factorization, we prove an improved discrete Rellich inequality and discuss its optimality. 
 We also give a conjecture on improved higher order discrete Hardy-like inequalities and formulate an open problem for the discrete Rellich inequality on the full space.
\end{abstract}
	
\maketitle

\section{Introduction}
The classical Hardy inequality asserts that
\begin{equation}\label{Hardy}
  \int_0^\infty |u'(x)|^2 \, \dd x
  \geq \frac{1}{4} \int_0^\infty \frac{|u(x)|^2}{x^2} \, \dd x ,
\end{equation}
for any $u \in H^1(0,\infty)$ with $u(0) = 0$.
Its discrete variant reads
\begin{equation}
 \sum_{n=1}^{\infty}|u_{n-1}-u_{n}|^{2}\geq\frac{1}{4}\sum_{n=1}^{\infty}\frac{|u_{n}|^{2}}{n^{2}},
\label{eq:hardy_ineq_clas}
\end{equation}
for any $u\in\ell^{2}(\N_{0})$ with $u_{0}=0$. 
These inequalities were established by G.~Hardy when trying to find a simple proof of the Hilbert inequality~\cite{Hardy-1920}. Other great mathematicians of the beginning of the 20th century, such as E.~Landau, G.~P{\'o}lya, I.~Schur, and M.~Riesz, also contributed to the development of Hardy's and related inequalities; see~\cite{Kufner-Maligranda-Person-2006} for a nice historical survey.

It is well known that the constant~$1/4$ on the right-hand side of~\eqref{Hardy} 
is the best possible and that the weight $1/(4x^2)$ 
cannot be replaced by a point-wise bigger function.
On the other hand,
while the constant $1/4$ on the right-hand side of~\eqref{eq:hardy_ineq_clas} 
is still the best possible, 
it was shown only recently by M.~Keller, Y.~Pinchover, and F. Pogorzelski 
in~\cite{Keller-Pinchover-Pogorzelski-2018-I} 
(see also \cite{Keller-Pinchover-Pogorzelski-2018-II} and preprint \cite{Fischer-Keller-Pogorzelski-2019-arxiv} for a $\ell^{p}$-generalization)
that the inequality can be improved by replacing the whole weight $1/(4n^{2})$ 
by a point-wise bigger sequence.
The improved discrete Hardy inequality reads
\begin{equation}
 \sum_{n=1}^{\infty}|u_{n-1}-u_{n}|^{2}\geq\sum_{n=1}^{\infty}\rho_{n}^{(1)}|u_{n}|^{2},
\label{eq:hardy_ineq_opt}
\end{equation}
where
\begin{equation}
\rho^{(1)}_{n}:=2-\sqrt{\frac{n-1}{n}}-\sqrt{\frac{n+1}{n}}>\frac{1}{4n^{2}},
\label{eq:hardy_weight_opt}
\end{equation}
for all $n\in\N$ (the reason for the upper index in $\rho^{(1)}$ will become clear later). An elementary proof of~\eqref{eq:hardy_ineq_opt} can be found in~\cite{Krejcirik-Stampach-2022}; see also~\cite{Huang-2021} for another proof.   

Moreover, the authors of~\cite{Keller-Pinchover-Pogorzelski-2018-II} 
proved that the Hardy weight $\rho^{(1)}$ 
cannot be replaced by a point-wise bigger sequence, in fact the weight
exhibits a strong optimality property. 
The notion of optimality has been established for Hardy inequalities in a different context of second order elliptic differential operators in~\cite{Devyver-Fraas-Pinchover-2014} and adjusted to the discrete setting of graphs in~\cite{Keller-Pinchover-Pogorzelski-2018-II}. 
To explain it in the present setting of discrete Hardy inequalities
on the half-line, let us denote 
\[
 H_0^M (\N_{0}):=\{ u \in \ell^2 (\N_0) \mid  u_0 = \dots = u_{M-1} = 0\}, \quad M\in\N.
\]
A positive sequence $\rho$ is called a \emph{Hardy weight}, if~\eqref{eq:hardy_ineq_opt} holds on $H_0^1(\N_{0})$ with $\rho^{(1)}$ being replaced by $\rho$. A Hardy weight $\rho$ is said to be \emph{optimal}, if the following three properties hold:
\begin{enumerate}
\item[1.] (\emph{Criticality}) 
The Hardy weight $\rho$ is \emph{critical}
meaning that
for any Hardy weight~$\tilde{\rho}$ such that $\tilde{\rho}\geq\rho$ point-wise, it follows that $\tilde{\rho}=\rho$.
\item[2.] (\emph{Non-attainability}) 
The Hardy weight $\rho$ is \emph{non-attainable}
meaning that
if the equality in~\eqref{eq:hardy_ineq_opt} is attained for a sequence $u$ such that the right-hand side of~\eqref{eq:hardy_ineq_opt} is finite, i.e.~$\sqrt{\rho}\,u\in\ell^{2}(\N)$, then $u=0$.
\item[3.] (\emph{Optimality near infinity}) 
The Hardy weight $\rho$ is \emph{optimal near infinity}
meaning that
for any $M\in\N$ and $\eps>0$, 
there exists a finitely supported sequence $u = u (M,\eps)\in H_0^M(\N_{0})$ such that
\begin{equation}
	\sum_{n=1}^\infty|u_{n-1}-u_{n}|^{2}< (1+\eps) \sum_{n=1}^\infty \rho_{n} |u_n|^2.
\end{equation}
\end{enumerate}

A simultaneous validity of properties~1 and~2 is called \emph{null-criticality in~\cite{Keller-Pinchover-Pogorzelski-2018-II}.}
The criticality of $\rho$ means that the respective Hardy inequality cannot be further improved. On the other hand, there exist infinitely many critical Hardy weights which are different from $\rho^{(1)}$, see~\cite[Thm.~10 and Cor.~12]{Krejcirik-Laptev-Stampach-2021-arxiv}. 
The non-attainability of $\rho$ means that, in the weighted $\ell^{2}$-space of sequences $u$ satisfying $u_{0}=0$ and $\sqrt{\rho}\,u\in\ell^{2}(\N_0)$, the inequality in~\eqref{eq:hardy_ineq_opt} is actually strict unless $u$ is trivial. Finally, the optimality of $\rho$ near infinity means that the constant $1$ in~\eqref{eq:hardy_ineq_opt} cannot be improved on the subspace
$H_0^M(\N_{0})$ for any $M\in \N$. Equivalently, the optimality of $\rho$ near infinity can be formulated as
\[
 \inf_{u\in H_0^M(\N_{0})\setminus\{0\}}
 \frac{\displaystyle \sum_{n=1}^\infty|u_{n-1}-u_{n}|^{2}}
 {\displaystyle \sum_{n=1}^\infty \rho_{n} |u_n|^2}=1,
\]
for all $M\in\N$.

There is an operator-theoretic interpretation of Hardy inequalities.
In the discrete setting, 
the Dirichlet Laplacian, conventionally denoted by $-\Delta$, 
is the second order difference operator
\begin{equation}
	\quad (-\Delta u)_{n}:= \begin{cases}
	&\hskip-8pt 2u_{0}-u_{1}, \quad\hskip12pt \mbox{ if }\; n=0,\\
	-u_{n-1}+&\hskip-8pt 2u_{n}-u_{n+1},  \quad \mbox{ if }\;  n\in\N,
	\end{cases}
\end{equation}
acting on the space of complex-valued sequences $\{u_n\}$ indexed by $n\in\N_{0}$. 
The discrete Hardy inequality \eqref{eq:hardy_ineq_opt} 
can be viewed as the inequality for the Laplacian $-\Delta\geq\rho^{(1)}$ on $H_0^1(\N_{0})$ understood in the sense of quadratic forms. 
Since $\rho^{(1)}$ is positive, the operator~$-\Delta$ 
is \emph{subcritical} in the language of~\cite{Keller-Pinchover-Pogorzelski-2018-II}.
Notice that the weight $\rho^{(1)}$ can be written as
\begin{equation}
 \rho^{(1)}=\frac{-\Delta g^{(1)}}{g^{(1)}}, \; \mbox{ where }\; g^{(1)}_{n}:=\sqrt{n}.
\label{eq:def_rho_1_g_1}
\end{equation}
This is not a coincidence, which will become clear in the proofs below, 
see also the methods presented in~\cite{Keller-Pinchover-Pogorzelski-2018-II}.

When the Laplacian $-\Delta$ is replaced by the bi-Laplacian $(-\Delta)^{2}$, 
we obtain a higher-order inequality commonly referred to as the \emph{Rellich inequality}. The classical continuous Rellich inequality appeared first in~\cite{Rellich-1956}. Its one-dimensional form on the half-line reads
\begin{equation}
 \int_{0}^{\infty}\left|u''(x)\right|^2\dd x\geq\frac{9}{16}\int_{0}^{\infty}\frac{|u(x)|^{2}}{x^{4}}\dd x,
\label{eq:rellich-ineq-cont}
\end{equation}
where $u\in H^{2}(0,\infty)$ with $u(0)=u'(0)=0$ 
(the form domain of the continuous Dirichlet bi-Laplacian,
whose quadratic form appears on the left-hand side). 
The main goal of this article is to establish a discrete analogue of the Rellich inequality and show that the direct discrete analogue of~\eqref{eq:rellich-ineq-cont} admits an improvement, 
similarly as in the case of the Hardy inequality. The notion of optimality can be defined for Rellich weights analogously as for Hardy weights replacing the subspace $H_0^1(\N_{0})$ by $H_0^2(\N_{0})$ and the Laplacian $-\Delta$ by the bi-Laplacian $(-\Delta)^{2}$. In contrast to the Hardy case, however, we were not able to prove (nor disprove) the criticality of the obtained improved Rellich weight. The remaining properties from the definition of the optimality hold. Our main result is the following theorem.

\begin{theorem}\label{thm:main}
 For all $u\in H_0^2(\N_{0})$, the discrete Rellich inequality
 \begin{equation}
  \sum_{n=1}^{\infty}\left|(-\Delta u)_{n}\right|^{2}\geq\sum_{n=2}^{\infty}\rho^{(2)}_{n}|u_{n}|^{2}
 \label{eq:rellich-ineq-disc-impr}
 \end{equation}
 holds with
 \begin{equation}
  \rho^{(2)}=\frac{(-\Delta)^{2}g^{(2)}}{g^{(2)}}, \; \mbox{ where }\; g^{(2)}_{n}=n^{3/2}. 
 \label{eq:def_rh_2_g_2}
 \end{equation}
 Moreover, the Rellich weight $\rho^{(2)}$ is 
 non-attainable and optimal near infinity.
 \end{theorem}

\begin{remark}$\ $
Notice that the squares of the sequences $g^{(1)}$ and $g^{(2)}$, i.e.~the sequences $\{n\}$ and $\{n^{3}\}$, are annihilated by $-\Delta$ and $(-\Delta)^{2}$, respectively.
\end{remark}

Explicitly, the improved Rellich weight reads
\[
 \rho^{(2)}_{n}=6-4\left(1+\frac{1}{n}\right)^{3/2}-4\left(1-\frac{1}{n}\right)^{3/2}+\left(1+\frac{2}{n}\right)^{3/2}+\left(1-\frac{2}{n}\right)^{3/2},
\]
for $n\geq 2$. It is straightforward to expand the above expression into a power series which yields
\begin{equation}\label{eq:asymp.rho.2}
\rho^{(2)}_{n}=6\sum_{l=1}^{\infty}\frac{4^{l}-1}{4^{2l}}\frac{(4l)!}{(2l)!(2l+2)!}\frac{1}{n^{2l+2}}=\frac{9}{16n^{4}}+\Oo\left(\frac{1}{n^{6}}\right).
\end{equation}
The coefficients of the power series expansion above are all positive, in particular implying that 
\begin{equation}
\rho^{(2)}_{n}>\frac{9}{16n^{4}},
\end{equation}
for all $n\geq2$. Hence, Theorem~\ref{thm:main} improves upon the direct discrete analogue of~\eqref{eq:rellich-ineq-cont}, i.e.~the inequality 
\begin{equation}
 \sum_{n=1}^{\infty}\left|(-\Delta u)_{n}\right|^{2}\geq\frac{9}{16}\sum_{n=2}^{\infty}\frac{|u_{n}|^{2}}{n^{4}},
\label{eq:rellich-ineq-disc}
\end{equation}
for all $u\in H_0^2(\N_{0})$. In fact, we are not aware of~\eqref{eq:rellich-ineq-disc} 
to be proved in the literature. Only recently, S.~Gupta obtained inequality~\eqref{eq:rellich-ineq-disc} with the constant $8/16$ instead of $9/16$ in the preprint~\cite{Gupta-2021-arxiv}. The more recent developments  of M.~Keller, Y.~Pinchover, and F.~Pogorzelski in~\cite{Keller-Pinchover-Pogorzelski-2021} on Rellich inequalities on graphs do not seem to cover our main result nor~\eqref{eq:rellich-ineq-disc} either, as they introduce additional weights into the inequalities which cannot be removed in the particular case of interest.

The paper is organized as follows. First, in Section~\ref{sec:hardy},
we briefly describe an elementary method of proving the optimal discrete Hardy inequality~\eqref{eq:hardy_ineq_opt}. 
The idea of the proof is implicitly behind the approach of
the second and third authors to establish the inequality,
as well as the criticality of $\rho^{(1)}$, 
in~\cite{Krejcirik-Stampach-2022}.
Our point here is to emphasize the common logic of our approach that admits a generalization to the case of the Rellich inequality. Moreover, we complete the proof of the optimality of~$\rho^{(1)}$ by showing the non-attainability and the optimality of $\rho^{(1)}$ near infinity, which have not been addressed in~\cite{Krejcirik-Stampach-2022}. 

Our method of proof is based on a factorization of $-\Delta - \rho^{(1)}$
and from this respect it resembles a general approach of F.~Gesztesy
and L.~L.~Littlejohn in the continuous setting in~\cite{Gesztesy-Littlejohn_2018}
(see also \cite{Gesztesy-Littlejohn-Michael-Wellman_2018}).
We are not aware of any development of the idea in the discrete setting.
The main novelty of the present paper is a successful factorization 
of the bi-Laplacian operator $(-\Delta)^2 - \rho^{(2)}$.
This is done in Section~\ref{sec:rellich}, 
where we apply the indicated method and prove Theorem~\ref{thm:main}.
In addition, for the criticality of~$\rho^{(2)}$ is not established, 
we prove a related result which shows that an improved Rellich weight cannot be too far from $\rho^{(2)}$ measured by a~suitable metric. 

Finally, in Section~\ref{sec:conj}, 
we provide a conjecture on discrete higher order Hardy-like inequalities for $(-\Delta)^{k}$, with $k\geq3$, and discuss open problems related to the discrete Rellich inequality. 
The paper is concluded by two appendices, where several auxiliary lemmas are proven.

\section{The Hardy inequality}\label{sec:hardy}

We motivate the Rellich inequality based on the Hardy case. 
We follow the method of proof of~\eqref{eq:hardy_ineq_opt}
from~\cite{Krejcirik-Stampach-2022}, which enables one to 
even determine the remainder term in the discrete Hardy inequality explicitly.

\subsection{Proof of the Hardy inequality}

It is sufficient to prove~\eqref{eq:hardy_ineq_opt} for compactly supported sequences in $H_0^1(\N_{0})$, the full statement then follows by a standard approximation argument. The method proceeds by making the ansatz that the operator $-\Delta-\rho^{(1)}$, which has a tri-diagonal matrix representation on $H_0^1(\N_{0})$ (with respect to the standard basis) decomposes as $R_{1}^{*}R_{1}$, where $R_{1}$ is determined by a real bi-diagonal matrix. It means that we suppose
\begin{equation}\label{eq.Hardy.ansatz}
	\sum_{n=1}^\infty |u_{n-1}-u_{n}|^2 = \sum_{n=1}^\infty \rho_{n}^{(1)} |u_n|^2 + \sum_{n=1}^\infty |(R_1 u)_n|^2,
\end{equation}
where
\begin{equation}\label{eq.R1}
	(R_1 u)_n := a_n u_n - \frac{1}{a_n} u_{n+1}, \quad n \in \N,
\end{equation}
and $\{a_n\}_{n \in \N}\subset\R\setminus\{0\}$ is an unknown sequence we seek.
Note that the coefficients of the remainder are necessarily of this reciprocal form due to the particular structure of the ansatz. If such $\{a_n\}$ exists, the Hardy inequality holds and the remainder is explicit in terms of the sequence $\{a_n\}$.

%\begin{remark}
%	Upon identification of $\Dd_1$ with $\ell^2(\N)$, inequality~\eqref{eq:hardy_ineq_opt} reads
%	\begin{equation}\label{eq.Hardy.form}
%		\langle -\Delta u, u \rangle = \| Du \|^2  \ge \| \sqrt{\rho^{(1)}} u\|^2 = \langle \rho^{(1)} u, u \rangle, \quad u \in \ell^2 (\N),
%	\end{equation}
%	with the norm and scalar product understood in $\ell^2 (\N)$, i.e.~$-\Delta \ge \rho^{(1)}$ on $\ell^2 (\N)$ in the form sense. Imposing~\eqref{eq.Hardy.ansatz} means assuming that one can factorise
%	\begin{equation}
%		-\Delta - \rho_1 = R_1^*R_1.
%	\end{equation}
%	The remainder in~\eqref{eq.Hardy} is then \footnote{I put $\ast$ to indicate the end of remarks. Better suggestion?}
%	\begin{equation}
%		\|Du \|^2 - \|\sqrt{\rho_1} u\|^2 = \langle R_1^* R_1 u, u \rangle = \| R_1 u\|^2 \ge 0. \tag*{$\ast$}
%	\end{equation}
%\end{remark}

Expanding~\eqref{eq.Hardy.ansatz}--\eqref{eq.R1} with the weight in~\eqref{eq:hardy_weight_opt} leads to the set of equations
\begin{equation}\label{eq.Hardy.set1}
	a_1^2 = \sqrt 2, \quad a_n^2 + \frac1{a_{n-1}^2} = \sqrt\frac{n+1}{n} + \sqrt\frac{n-1}{n}, \quad n \ge 2.
\end{equation}
As a first order difference equation, it can be solved explicitly and has the unique positive solution
\begin{equation}
	a_n = \sqrt[4]{\frac{n+1}{n}}, \quad n \in \N.
\end{equation}
Thus, identity~\eqref{eq.Hardy.ansatz} holds true with the remainder term determined by the operator
\begin{equation}
	(R_1u)_{n}=\sqrt[4]{\frac{n+1}{n}} u_n - \sqrt[4]{\frac{n}{n+1}} u_{n+1}, \quad n\in\N,
	\label{eq:R_1_explicitly}
\end{equation}
and the discrete Hardy inequality~\eqref{eq:hardy_ineq_opt} follows.

\begin{remark}\label{rem:R_1_anihil_g_1}
 Notice that $R_{1}$ annihilates the sequence $g^{(1)}_{n}=\sqrt{n}$, see~\eqref{eq:def_rho_1_g_1}, i.e. $R_{1}g^{(1)}=0$. 
\end{remark}

\begin{remark}
Alternatively, identity~\eqref{eq.Hardy.ansatz} can be written as the operator equality
\[
 -\Delta-\rho^{(1)}=R_{1}^{*}R_{1}
\]
on the subspace $H_0^1(\N_{0})$ of $\ell^{2}(\N_{0})$. It implies the Hardy inequality, which can be written as the operator inequality $-\Delta\geq\rho^{(1)}$ on $H_0^1(\N_{0})$ understood in the sense of quadratic forms.
\end{remark}

\begin{remark}
In fact, any positive sequence $\{g_{n}\}_{n \in \N}$ together with $g_{0} = 0$ such that $\rho:=(-\Delta g)/g$ becomes point-wise positive results in a Hardy inequality with the remainder in~\eqref{eq.R1} and $a_n = \sqrt{g_{n+1} / g_{n}}$, see~\cite[Thm.~10]{Krejcirik-Laptev-Stampach-2021-arxiv}. The authors therein also give a sufficient condition for the criticality of this more general Hardy weight.
\end{remark}

\subsection{Optimality of the Hardy inequality}

Next, we prove the optimality of the Hardy weight~\eqref{eq:hardy_weight_opt}, which was originally shown in a more general graph setting in~\cite{Keller-Pinchover-Pogorzelski-2018-II}. We present another proof, making use of the concrete remainder determined by~\eqref{eq:R_1_explicitly} and suitable regularizing sequences whose choice is motivated by standard proofs of the continuous Hardy inequality. Recall that the notion of optimality comprises criticality, non-attainability and optimality of $\rho^{(1)}$ near infinity.

\subsubsection{Criticality}

To demonstrate the criticality of $\rho^{(1)}$, we need to check that the discrete Hardy inequality cannot hold with any strictly larger weight. Suppose $\rho$ is a Hardy weight such that $\rho_{n}\geq\rho^{(1)}_{n}$ for all $n\in\N$.
Then, by~\eqref{eq.Hardy.ansatz},
\begin{equation}
 0\leq\sum_{n=1}^{\infty}(\rho_{n}-\rho^{(1)}_{n})|u_{n}|^{2}\leq \sum_{n=1}^\infty \left|(R_{1}u)_{n} \right|^2,
\label{eq:towards_optimality}
\end{equation}
for all compactly supported sequences $u$. Observing that the right-hand side of~\eqref{eq:towards_optimality} vanishes for $u=g^{(1)}$, which, however, is not a finitely supported sequence, the idea is to conveniently regularize $g^{(1)}$. A possible regularization introduces $u_n^N := \xi_n^N \sqrt{n}$, where
\[
  \xi_n^N :=
  \begin{cases}
    1 & \mbox{if} \quad n < N \,,
    \\[2pt]
    \frac{\displaystyle 2 \log{N}-\log n}{\displaystyle\log N}
    & \mbox{if} \quad N\leq n \leq N^2 \,,
    \\[2pt]
    0 & \mbox{if} \quad n > N^2 \,,
  \end{cases}
\]
with $N\geq 2$. Notice that $\xi^N\to1$ point-wise as $N \to \infty$. Further, with the aid of formula~\eqref{eq:R_1_explicitly}, we may plug $u^{N}$ into~\eqref{eq:towards_optimality} and deduce the upper bound
\begin{align*}
\sum_{n=1}^\infty \left|(R_{1}u^N)_{n} \right|^2&=\sum_{n=1}^\infty\sqrt{n(n+1)}\left|\xi_{n+1}^{N}-\xi_{n}^{N}\right|^{2}\\
&=\frac{1}{\log^{2}N}\sum_{n=N+1}^{N^{2}}\sqrt{n(n-1)}\log^{2}\left(\frac{n}{n-1}\right)\\
&\leq\frac{2}{\log^{2}N}\sum_{n=N+1}^{N^{2}}\frac{1}{n-1}\leq\frac{4}{\log N}
\end{align*}
for its right-hand side. As the last expression tends to zero, as $N\to\infty$, by monotone convergence we arrive at the equality 
\[
\sum_{n=1}^{\infty}n(\rho_{n}-\rho^{(1)}_{n})=0.
\]
Since each term in the last series is non-negative, we conclude  $\rho=\rho^{(1)}$.

\subsubsection{Non-attainability}

One observes from~\eqref{eq.Hardy.ansatz} that the discrete Hardy inequality holds as an equality for a sequence $u$ with $u_{0}=0$, if and only if, $R_{1}u=0$. Taking~\eqref{eq:R_1_explicitly} into account, it is clear that the solution of $R_{1}u=0$ equals $g^{(1)}$ up to a multiplicative constant, i.e. $u_{n}=c\sqrt{n}$ for $c\in\C$ and $n\in\N_{0}$. Since $\rho_{n}^{(1)}>1/(4n^{2})$, for all $n\in\N$, one sees that 
\[
 \sum_{n=1}^{\infty}\rho_{n}^{(1)}\left|u_{n}\right|^{2}=\infty,
\]
whenever $c\neq0$. 
This result together with the already established criticality implies
that $\rho^{(1)}$ is actually null-critical.

\subsubsection{Optimality near infinity}

As it follows from the definition and the identity~\eqref{eq.Hardy.ansatz}, it suffices, for $M\in\N$ fixed, to find a sequence of elements $u^{N}\in H_0^M(\N_{0})\setminus\{0\}$ such that 
\[
 \lim_{N\to\infty}
 \frac{\displaystyle\sum_{n=1}^{\infty}\left|R_{1}u^{N}_{n}\right|^{2}}
 {\displaystyle\sum_{n=1}^{\infty}\rho_{n}^{(1)}\left|u^{N}_{n}\right|^{2}}=0.
\]
We can take take a slight modification of the sequence in the proof of the criticality of $\rho^{(1)}$. More precisely, we consider~$\tilde u^{N}:=\tilde \xi^{N}g^{(1)}$, for $N \ge M$, with
\[
\tilde \xi_n^N :=
\begin{cases}
	0 & \mbox{if} \quad n < N \,,
	\\[2pt]
	\frac{\displaystyle  \log{n}-\log N}{\displaystyle\log N} & \mbox{if} \quad N\leq n \leq N^2 \,, \\[2pt]
	1 & \mbox{if} \quad N^2< n < 2N^2 \,,
	\\[2pt]
	\frac{\displaystyle \log{(2N^3)}-\log n}{\displaystyle\log N}
	& \mbox{if} \quad 2N^2 \leq n \leq 2N^3 \,,
	\\[2pt]
	0 & \mbox{if} \quad n > 2N^3 \,.
\end{cases}
\]
While the proof of  $\|R_{2}\tilde u^{N}\|^{2}=\Oo(1/\log N)$ is analogous as in the proof of the criticality, one uses $\rho^{(1)}_n > 1/(4n^2)$, for all $n \in \N$, to bound the denominator from below by
\begin{equation}
	\sum_{n=1}^{\infty}\rho_{n}^{(1)}\left|u^{N}_{n}\right|^{2}>\sum_{n=N^2}^{2N^2}\rho_{n}^{(1)}\left|u^{N}_{n}\right|^{2}=\sum_{n=N^2}^{2N^2}\rho_{n}^{(1)}n>\frac{1}{4}\sum_{n=N^2}^{2N^2}\frac{1}{n}>\frac{1}{4}\log 2,
\end{equation}
for all $N\geq2$. The result readily follows.

\section{The Rellich inequality}\label{sec:rellich}

The goal is to generalize the approach from the previous section to deduce a lower bound for the square of $-\Delta$ on the subspace $H_0^2(\N_{0}) = \left\{ u \in \ell^2 (\N_0) \, : \, u_0 = u_1 = 0 \right\}$ and prove the Rellich inequality of Theorem~\ref{thm:main}.

\subsection{Proof of the Rellich inequality}

Similarly as in the Hardy case, we seek a remainder term so that the identity 
\begin{equation}\label{eq.Rellich.ansatz}
	 \sum_{n=1}^{\infty} |(-\Delta u)_n|^2 = \sum_{n=2}^{\infty} \rho_{n}^{(2)} |u_n|^2 + \sum_{n=1}^{\infty} |(R_2 u)_n|^2
\end{equation}
holds for all $u\in H_0^2(\N_{0})$, where $\rho^{(2)}$ is defined by~\eqref{eq:def_rh_2_g_2}.  We suppose that $R_2$ can be found as a second order difference operator. More precisely, the ansatz is that $R_{2}$ acts on sequences by the formula
\begin{equation}\label{eq.R2}
	(R_2 u)_n := c_n u_n - b_n u_{n+1} + \frac{1}{c_n} u_{n+2}, \quad n \in \N,
\end{equation}
where $\{b_n\}_{n\in\N}\subset\R$ and $\{c_n\}_{n\in\N}\subset\R\setminus\{0\}$. Note that starting with a more general ansatz of a second order difference operator $R_{2}$ with three unknown coefficients results in this special form with the first and the third coefficients reciprocal, reducing the number of unknown sequences from three to two. Below, we prove that there indeed exist sequences $\{b_{n}\}$ and $\{c_{n}\}$ with positive entries such that~\eqref{eq.Rellich.ansatz} and~\eqref{eq.R2} hold. 
Unlike in the Hardy case, however, we do not obtain a fully explicit formula for the remainder due to the second order differences in the underlying system of equations. 

\begin{remark}
	The claimed Rellich inequality is in fact a lower bound $(-\Delta)^2 \ge \rho^{(2)}$ for the bi-Laplacian on the subspace $H_0^2(\N_{0})$ in the form sense. Analogously to the Hardy case, our ansatz assumes the factorization
	\begin{equation}
		(-\Delta)^2 - \rho^{(2)} = R_2^* R_2
	\end{equation}
	on $H_0^2(\N_{0})$, which leads to the Rellich inequality
	\begin{equation}
		\| -\Delta u \|^2 -  \left\|\sqrt{\rho^{(2)}} u \right\|^2 = \|R_{2}u\|^2 \ge 0.
	\end{equation}
\end{remark}

The proof of the existence of point-wise positive sequences $\{b_{n}\}$ and $\{c_{n}\}$ such that~\eqref{eq.Rellich.ansatz} and~\eqref{eq.R2} hold relies on a lemma for an auxiliary sequence defined recursively by
\begin{equation}\label{eq:def_xi}
	\zeta_{n} := \frac{g_{n+1}^{(2)}}{g_{n}^{(2)}} \left( 4 - \frac{g_{n+2}^{(2)}}{g_{n+1}^{(2)}} - \frac{g_{n-1}^{(2)}}{g_{n}^{(2)}} - \frac{g_{n+1}^{(2)}}{g_{n}^{(2)}} \frac{1}{\zeta_{n-1}} \right), \quad n \geq 2,
\end{equation}
and the initial condition 
\begin{equation}\label{eq:def_xi1}
	\zeta_{1} := \frac{4g_{2}^{(2)} - g_{3}^{(2)}}{g_{1}^{(2)}},
\end{equation}
where $g^{(2)}_{n}=n^{3/2}$, see~\eqref{eq:def_rh_2_g_2}. It turns out that the sequence $\zeta$ is well defined, i.e. $\zeta_{n}\neq0$ for all $n\in\N$. In fact, we will need more detailed upper and lower bounds on $\zeta_{n}$ for later purposes, which is the content of the following lemma. As its proof is unrelated to the presented method for the derivation of the Rellich inequality, we postpone it to Appendix~\ref{subsec:a.1} 
in order not to distract from our primary intention.

\begin{lemma}\label{lem.xi.bounds}
	For all $n\in\N$, the solution of~\eqref{eq:def_xi} and~\eqref{eq:def_xi1} satisfies
	\begin{equation}
		\left(1 + \frac{2}{n}\right)^{3/2} < \zeta_{n} < \left(1 + \frac{3}{n}\right)^{3/2}.
	\end{equation}
	In particular, $\zeta_{n}>0$ for all $n\in\N$.
\end{lemma}

Now we can prove the Rellich inequality~\eqref{eq:rellich-ineq-disc-impr}. It suffices to prove it for finitely supported sequences $u \in H_0^2(\N_{0})$. The full claim then follows from a standard approximation argument.
	
By construction,~\eqref{eq:rellich-ineq-disc-impr} follows from the existence of positive sequences $\{b_n\}$ and $\{c_n\}$ such that the remainder ansatz~\eqref{eq.Rellich.ansatz}--\eqref{eq.R2} holds. To show that such sequences exist, we first expand
\begin{align*}
	|(-\Delta u)_{n}|^{2}&= |u_{n-1} - 2u_n + u_{n+1}|^2\\
			&= |u_{n-1}|^2 + 4  |u_n|^2 +  |u_{n+1}|^2 -4 \Re (u_{n-1} \bar u_n +  u_n \bar u_{n+1} ) +2 \Re (u_{n-1} \bar u_{n+1}).
\end{align*}
	Using that $u_0 = u_1 = 0$, we see that
	\begin{equation} \label{eq.Delta.u}
		\| -\Delta u \|^2 =  6 \sum_{n=2}^{\infty} |u_n|^2 - 8 \Re \sum_{n=2}^{\infty} u_n \bar u_{n+1} + 2 \Re \sum_{n=2}^{\infty} u_n \bar u_{n+2}.
	\end{equation}
	On the other hand,
	\begin{equation} 
		\begin{aligned}
			|c_n u_n - b_n u_{n+1} + \frac{1}{c_n} u_{n+2}|^2 = c_n^2 & |u_n|^2 + b_n^2 |u_{n+1}|^2 + \frac{1}{c_n^2} |u_{n+2}|^2 \\
			& - 2 \Re \left( c_n b_n u_n \bar u_{n+1} + \frac{b_n}{c_n} u_{n+1} \bar u_{n+2} - u_n \bar u_{n+2} \right).
		\end{aligned}
	\end{equation}
	Since $u_0 = u_1 = 0$, this gives
	\begin{equation}\label{eq.R2.u}
		\begin{aligned}
			\| R_2u \|^2 = \sum_{n=3}^{\infty} & \left( c_n^2  + b_{n-1}^2 + \frac{1}{c_{n-2}^2} \right) |u_n|^2 + (c_2^2 + b_1^2) |u_2|^2 \\
			& - 2 \Re \sum_{n=2}^{\infty} \left( c_n b_n + \frac{b_{n-1}}{c_{n-1}} \right) u_n \bar u_{n+1} + 2 \Re \sum_{n=2}^{\infty} u_n \bar u_{n+2}.
		\end{aligned}
	\end{equation}
	After inserting~\eqref{eq.Delta.u} and~\eqref{eq.R2.u} into~\eqref{eq.Rellich.ansatz}, we arrive at the following two countable sets of equations
	\begin{equation}\label{eq.set1}
			6  = \rho_{2}^{(2)} + c_2^2 + b_1^2, \quad 
			6  = \rho_{n}^{(2)} + c_n^2 + b_{n-1}^2 + \frac{1}{c_{n-2}^2}, \quad n \ge 3, 
	\end{equation}
	and secondly
	\begin{equation}\label{eq.set2}
		4  = c_n b_n + \frac{b_{n-1}}{c_{n-1}}, \quad n \ge 2.
	\end{equation}
	
Next, we impose the additional condition that $g^{(2)}$ is annihilated by the remainder operator $R_2$, i.e.~$R_{2}g^{(2)}=0$; cf.~Remark~\ref{rem:R_1_anihil_g_1}. It means that we assume the third set of equations
	\begin{equation}\label{eq.set3}
		c_{n}g_{n}^{(2)}-b_{n}g_{n+1}^{(2)}+\frac{1}{c_{n}}g_{n+2}^{(2)}=0, \quad n\in\N.
	\end{equation}
The task is to find sequences $\{b_n\}$ and $\{c_n\}$ which solve the equations~\eqref{eq.set1},~\eqref{eq.set2}, and~\eqref{eq.set3} simultaneously.
	
	If we reduce $b_{n}$ from~\eqref{eq.set3} and use it in~\eqref{eq.set2}, we arrive at the following recurrence 
	\begin{equation}\label{eq.c.recur}
		\frac{g_{n}^{(2)}}{g_{n+1}^{(2)}}c_{n}^{2}+\frac{g_{n+1}^{(2)}}{g_{n}^{(2)}}\frac{1}{c_{n-1}^{2}}=4-\frac{g_{n+2}^{(2)}}{g_{n+1}^{(2)}}-\frac{g_{n-1}^{(2)}}{g_{n}^{(2)}}, \quad n\geq2.
	\end{equation}
	When fixing an initial value, this determines the sequence $\{c_{n}^{2}\}$ uniquely, provided that $c_{n}^{2}$ never vanishes. We choose
	\begin{equation}\label{eq.c1}
		c_1^2 := \frac{4g_{2}^{(2)}-g_{3}^{(2)}}{g_{1}^{(2)}}
	\end{equation}
	and apply Lemma~\ref{lem.xi.bounds} with $\zeta_n := c_n^2$ to conclude that the resulting sequence is indeed positive. In the sequel, $\{c_n\}$ shall be the sequence of positive square roots and $\{b_n\}$ shall be computed from~\eqref{eq.set3}. Then~\eqref{eq.set2} and~\eqref{eq.set3} hold by construction.
	
	Moreover, these solutions are compatible with~\eqref{eq.set1}. In fact, for $n \ge 3$, if one computes $b_{n-1}$ from~\eqref{eq.set3}, expresses $c_{n-2}^{2}$ and $c_{n}^{2}$ in terms of $c_{n-1}^{2}$ employing~\eqref{eq.c.recur} and uses
	\begin{equation}\label{eq.rho2.frac}
		\rho_{n}^{(2)} = \frac{g_{n-2}^{(2)} - 4 g_{n-1}^{(2)} + 6 g_{n}^{(2)} - 4 g_{n+1}^{(2)} + g_{n+2}^{(2)}}{g_{n}^{(2)}}, \quad n \ge 2,
	\end{equation}
	as in~\eqref{eq:def_rh_2_g_2}, one readily checks that the second equation in~\eqref{eq.set1} holds identically. The first equation in~\eqref{eq.set1} can be verified similarly by expressing $b_1$ from~\eqref{eq.set3}, $c_2^2$ from~\eqref{eq.c.recur} and using the definition of $c_1^2$ in~\eqref{eq.c1} and $\rho_{2}^{(2)}$ in~\eqref{eq.rho2.frac}.
	
	In summary, we have proved that the required sequences exist and hence identity~\eqref{eq.Rellich.ansatz} is established.

\subsection{On the optimality of $\rho^{(2)}$}

In this section, we prove the second claim of Theorem~\ref{thm:main} by verifying that the Rellich weight $\rho^{(2)}$ fulfills the non-attainability, 
as well as the optimality near infinity. Even though the remainder of our Rellich inequality is not fully explicit, the asymptotic information from Lemma~\ref{lem.xi.bounds} about the sequence $\{c_n^2\}$ proves to be sufficient.

Since the criticality of $\rho^{(2)}$ is not established, it is possible that the Rellich weight $\rho^{(2)}$ can be further improved. To this end, we prove another result which provides certain limits to a possible improvement. More precisely, we show that any Rellich weight, which improves upon $\rho^{(2)}$, lies in a ball centered at $\rho^{(2)}$ of a specific radius with respect to a suitable weighted $\ell^{1}$-metric.

\subsubsection{Non-attainability}

We prove that, if a sequence $u$ with $u_{0}=u_{1}=0$ satisfies
\[
 \sum_{n=1}^{\infty} |(-\Delta u)_n|^2 = \sum_{n=2}^{\infty} \rho_{n}^{(2)} |u_n|^2<\infty,
\]
then $u=0$. Taking~\eqref{eq.Rellich.ansatz} into account, the equality in the Rellich inequality implies that $R_{2}u=0$, which by~\eqref{eq.R2} is
\[
 c_{n}^{2}u_{n}+b_{n}c_{n}u_{n+1}+u_{n+2}=0, \quad n \in \N.
\]
Due to the requirement $u_{1}=0$, a general solution of this recurrence is determined uniquely up to a multiplicative constant. Using~\eqref{eq.set3}, we can reduce $b_{n}$ from the last equation and get the recurrence
	\[
	c_{n}^{2}\left(u_{n}-\frac{g_{n}^{(2)}}{g_{n+1}^{(2)}}u_{n+1}\right)-\frac{g_{n+2}^{(2)}}{g_{n+1}^{(2)}}\left(u_{n+1}-\frac{g_{n+1}^{(2)}}{g_{n+2}^{(2)}}u_{n+2}\right)=0, \quad n \in \N.
	\]
 This is a first-order difference equation for the unknown sequence
	\[
	v_{n}:=u_{n}-\frac{g_{n}^{(2)}}{g_{n+1}^{(2)}}u_{n+1}, \quad n \in \N.
	\]
	We can easily solve it by iteration and, if we additionally set $u_{2}=1$, as the solution is determined up to a multiplicative constant, we arrive at the formula 
	\[
	v_{n}=-\frac{g_{1}^{(2)}}{g_{n+1}^{(2)}}\prod_{j=1}^{n-1}c_{j}^{2}, \quad n \in \N.
	\]
	Solving the resulting first-order recurrence for $u$ yields the solution
	\begin{equation}
	u_{n}=g_{1}^{(2)}g_{n}^{(2)}\sum_{k=1}^{n-1}\frac{1}{g_{k}^{(2)}g_{k+1}^{(2)}}\prod_{j=1}^{k-1}c_{j}^{2}, \quad n \in \N.
	\label{eq:second_sol_R_2}
	\end{equation}
Thus, any solution of $R_{2}u=0$ with $u_0 = u_1 = 0$ is a constant multiple of~\eqref{eq:second_sol_R_2}.
	
	Notice that the solution~\eqref{eq:second_sol_R_2} is positive. Using the lower bound on $\zeta_{n}\equiv c_{n}^{2}$ from Lemma~\ref{lem.xi.bounds} and inserting $g_{n}^{(2)}=n^{3/2}$, we get the estimate
	\[
	u_{n}\geq n^{3/2} \sum_{k=1}^{n-1}\frac{1}{k^{3/2}(k+1)^{3/2}}\prod_{j=1}^{k-1}\left(1+\frac{2}{j}\right)^{3/2}=\frac{n^{3/2}(n-1)}{2^{3/2}}, \quad n\in\N.
	\] 
	From this it follows that $u_{n}\gtrsim n^{5/2}$, for all $n\ge 2$, where $\gtrsim$ is the inequality $\geq$ up to a universal multiplicative constant. Recalling that $\rho^{(2)}_{n}>9/(16n^{4})\gtrsim 1/n^{4}$, for all $n\geq2$, we see that $\rho_{n}^{(2)}|u_{n}|^{2}\gtrsim n$, hence it is not a summable sequence.

\subsubsection{Optimality near infinity}

We prove that, for any $M\geq2$, one has  
\begin{equation}\label{eq.opt3}
	\inf_{u \in H_0^M(\N_{0}) \setminus \{0\}} \frac{\langle (-\Delta)^2 u, u \rangle}{\langle \rho^{(2)}u,u\rangle}=1.
\end{equation}
Let $M\geq2$ be fixed. In view of~\eqref{eq.Rellich.ansatz}, we establish~\eqref{eq.opt3} by finding a sequence $\{u^{N}\}_{N\ge M}\subset H_0^M(\N_{0})$ such that 
\[
 \lim_{N\to\infty}\frac{\left\|R_{2}u^{N}\right\|^{2}}{\langle \rho^{(2)}u^N,u^N\rangle}=0.
\]
To define the sequence of elements $u^{N}$, we implement a suitable cut-off strategy for the sequence $g^{(2)}$, which is annihilated by $R_2$ per construction. With this choice, we show in two steps that 
\begin{equation}
\left\|R_{2}u^{N}\right\|^{2}\lesssim\frac{1}{\log N} \quad\mbox{ and }\quad \langle \rho^{(2)}u^N,u^N\rangle\gtrsim1,
\label{eq:remainde_estim}
\end{equation}
from which the statement readily follows.

1) \emph{The first inequality in~\eqref{eq:remainde_estim}}: We define $u^N := g^{(2)} \xi^N$ with the cut-off sequence 
\begin{equation}
	\xi_n^N := f^N(n), \quad n \in \N_0,
\end{equation}
where $f_n : [0,\infty) \to [0,1]$ is defined as
\begin{equation}
	f^N (x) := \begin{cases}
		0 & \mbox{if} \quad x < N \,,
		\\[2pt]
		\eta \left( \frac{\log x - \log N}{\log N} \right) & \mbox{if} \quad x \in [N, N^2) \,, \\[2pt]
		1 & \mbox{if} \quad x \in [N^2, 2N^2] \,, \\[2pt]
		\eta \left( \frac{\log (2N^3) - \log x}{\log N} \right) & \mbox{if} \quad x \in (2N^2, 2N^3] \,, \\[2pt]
		0 & \mbox{if} \quad x > 2N^3\,.
	\end{cases}
\end{equation}
Here $\eta$ is a smooth, real-valued, increasing function on $\R$ such that $\eta \equiv 0$ on $(-\infty, \eps)$ and $\eta \equiv 1$ on $(1-\eps, \infty)$ with a fixed $\eps > 0$. By this definition, $f^N$ is a smooth, compactly supported function. Clearly, $u^{N}\in H_0^M(\N_{0})$ for all $N\ge M$.

We proceed by using that $\{b_n\}$ and $\{c_n\}$ are constructed to satisfy~\eqref{eq.set3} and write the entries of the remainder term as
\begin{equation}
	\begin{aligned}
		(R_2 u^N)_n & = c_n g^{(2)}_n \xi^N_n - b_n g^{(2)}_{n+1} \xi^N_{n+1} + \frac{1}{c_n} g^{(2)}_{n+2} \xi^N_{n+2} \\
		& = c_n g^{(2)}_n \left(\xi^N_n - \xi^N_{n+1}\right) - \frac{1}{c_n} g^{(2)}_{n+2} \left(\xi^N_{n+1} - \xi^N_{n+2}\right).
	\end{aligned}
\end{equation}
From Lemma~\ref{lem.xi.bounds}, we know that
\begin{equation}
	1 < c_n = \sqrt \zeta_n < \left( 1 + \frac3n \right)^{3/4} = 1 + \mathcal O \left( \frac1n \right), \quad n \to \infty.
\end{equation}
Hence, one can write $c_n = 1 + p_n$, and thus $1/c_n = 1 + q_n$, where both $p_n$ and $q_n$ are of order $\mathcal O (1/n)$. Then the remainder can be written as
\begin{equation}\label{RuN}
	\begin{aligned}
		(R_2 u^N)_n & = - (g^{(2)}_{n+2} - g^{(2)}_n) \left(\xi^N_n - \xi^N_{n+1}\right) - g^{(2)}_{n+2} \left( - \xi^N_{n} + 2\xi^N_{n+1} - \xi_{n+2}^N \right) \\
		& \hskip67pt  + p_n g^{(2)}_n \left(\xi^N_n - \xi^N_{n+1}\right) - q_n g^{(2)}_{n+2} \left(\xi^N_{n+1} - \xi^N_{n+2}\right).
	\end{aligned}
\end{equation}
Using that
\begin{equation}
	g^{(2)}_{n+2} - g^{(2)}_n \le \frac32 \sqrt{n+2} \lesssim \sqrt{n}, \quad p_n \lesssim \frac1n, \quad q_n \lesssim \frac1n,
\end{equation}
it follows that
\begin{equation}\label{eq.RuN.asymp}
	\begin{aligned}
		(R_2 u^N)_n^2 & \lesssim n \left( \left(\xi^N_n - \xi^N_{n+1}\right)^2 + \left(\xi^N_{n+1} - \xi^N_{n+2}\right)^2 \right) \\
		& \qquad \qquad \qquad \qquad + n^3 \left( - \xi^N_{n} + 2\xi^N_{n+1} - \xi_{n+2}^N \right)^2.
	\end{aligned}
\end{equation}
To get from~\eqref{RuN} to~\eqref{eq.RuN.asymp}, we have also used that the square of a sum can be bounded by the sum of the squares by iterating the inequality $(a+b)^2 \le 2 (a^2 + b^2)$ for non-negative numbers $a,b$. To estimate the differences above, we derive 
\begin{equation}
	(f^N)' (x) = \begin{cases}
		\eta' \left( \frac{\log x - \log N}{\log N} \right) \frac{1}{x \log N} &  \mbox{if} \quad x \in [N, N^2] \,, \\[2pt]
		- \eta' \left( \frac{\log (2N^3) - \log x}{\log N} \right) \frac{1}{x \log N} &  \mbox{if} \quad x \in [2N^2, 2N^3]\,,
	\end{cases}
\end{equation}
and
\begin{equation}
	(f^N)'' (x) = \begin{cases}
		- \eta' \left( \frac{\log x - \log N}{\log N} \right) \frac{1}{x^2 \log N}  \\[2pt]
		\quad \qquad + \eta'' \left( \frac{\log x - \log N}{\log N} \right) \frac{1}{x^2 \log^2 N} & \mbox{if} \quad x \in [N, N^2] \,, \\[2pt]
		\eta' \left( \frac{\log (2N^3) - \log x}{\log N} \right) \frac{1}{x^2 \log N} \\[2pt]
		\quad \qquad + \eta'' \left( \frac{ \log (2N^3) - \log x}{\log N} \right) \frac{1}{x^2 \log^2 N} & \mbox{if} \quad x \in [2N^2, 2N^3]
		\,.
	\end{cases}
\end{equation}
This implies that
\begin{equation}\label{dfN.bdd}
	|(f^N)' (x)| \lesssim \frac{1}{x \log N} \quad\mbox{ and }\quad |(f^N)'' (x)| \lesssim \frac{1}{x^2 \log N},
\end{equation}
for all $x \in \R_+$ and $N \in \N$ (while for most $x$ the above are actually zero). Hence, for all $n, N \in \N$ we have
\begin{equation}\label{diff.est}
	|\xi^N_n - \xi^N_{n+1}| \lesssim \frac{1}{n \log N} \quad\mbox{ and }\quad |- \xi^N_{n} + 2\xi^N_{n+1} - \xi_{n+2}^N| \lesssim \frac{1}{n^2 \log N},
\end{equation}
which can be seen by the mean value theorem. We combine~\eqref{eq.RuN.asymp} with~\eqref{diff.est} to obtain
\begin{equation} \label{eq.RuN.bound}
	(R_2 u^N)_n^2 \lesssim \frac{1}{n \log^2 N}, \quad n, N \in \N.
\end{equation}
We use this estimate and that $f^N$ vanishes on $(0,N]$ and $[2N^3,\infty)$ to conclude
\begin{equation}
	\begin{aligned}
		\left\|R_{2}u^{N}\right\|^{2} & = \sum_{n=N-1}^{N^2-1} (R_2 u^N)_n^2 + \sum_{n=2N^2-1}^{2N^3-1} (R_2 u^N)_n^2\\
		& \lesssim \frac{1}{\log^2 N} \left( \sum_{n=N-1}^{N^2-1} \frac1n + \sum_{n=2N^2-1}^{2N^3-1} \frac1n \right) \\
		& \le \frac{1}{\log^2 N} \left( \int_{N-2}^{N^2-1} \frac1n {\rm d} n + \int_{2N^2-2}^{2N^3-1} \frac1n {\rm d} n \right) \\
		& = \frac{1}{\log^2 N} \log \left( \frac{(N^2-1) (2N^3 - 1)}{(N-2) (2N^2-2)} \right) 
		 \lesssim \frac{1}{\log N}.
	\end{aligned}
\end{equation}
Thus, the first inequality in~\eqref{eq:remainde_estim} is established.

2) \emph{The second inequality in~\eqref{eq:remainde_estim}}:
We use only terms where the cut-off function $\xi^N$ is one by definition (hence $u^N_n = g^{(2)}_{n}=n^{3/2}$ for such $n$) and estimate
	\begin{equation}\label{eq.low.bdd}
		\langle \rho^{(2)}u^N,u^N\rangle = \sum_{n=2}^\infty \rho_{n}^{(2)}|u^N_n|^2 \ge \sum_{n=N^2}^{2N^2} \rho_{n}^{(2)}n^3.
	\end{equation}
	Since the weight satisfies $\rho_n^{(2)}>9/(16n^4)$, for $n\ge 2$, we further obtain
	\begin{equation}
		\| \sqrt {\rho_2} u^N\|^2 \gtrsim \sum_{n=N^2}^{2N^2} \frac1n \ge \log 2.
	\end{equation}
	This, however, implies the second inequality in~\eqref{eq:remainde_estim}.

\subsubsection{An upper bound on a distance of an improved Rellich weight from $\rho^{(2)}$}

Recall that a positive sequence $\{\rho_n\}_{n\geq2}$ is called a Rellich weight, if and only if, the Rellich inequality~\eqref{eq:rellich-ineq-disc-impr} holds on $H_0^2(\N_{0})$ with $\rho^{(2)}$ replaced by $\rho$. We show that, if $\rho$ is a Rellich weight which is point-wise greater or equal to~$\rho^{(2)}$, then the distance of $\rho$ and $\rho^{(2)}$ in a~suitable metric is bounded from above by an explicit constant.

\begin{theorem}
 Let $\rho=\{\rho_{n}\}_{n\geq2}$ be a Rellich weight. If $\rho\geq\rho^{(2)}$, then 
 \[
  \sum_{n=2}^{\infty}n^{3}\!\left(\rho_{n}-\rho_{n}^{(2)}\right)\leq8\sqrt{2}-3\sqrt{3}.
 \]
\end{theorem}

\begin{proof}
 We show the claim by a slight modification of the proof of the first inequality in~\eqref{eq:remainde_estim}. Making use of the modified cut-off sequence
 \begin{equation}
	\tilde{f}^N (x) := \begin{cases}
		0 & \mbox{if} \quad x \in [0, 1] \,, \\[2pt]
		1 & \mbox{if} \quad x \in (1, 2N^2] \,, \\[2pt]
		\eta \left( \frac{\log (2N^3) - \log x}{\log N} \right) & \mbox{if} \quad x \in (2N^2, 2N^3] \,, \\[2pt]
		0 & \mbox{if} \quad x > 2N^3 \,,
	\end{cases}
\end{equation}
one obtains that
\[
 \|R_{2}\tilde{u}^{N}\|^{2}=c_{1}^{2}+\Oo \left(\frac{1}{\log N}\right), \quad N\to\infty,
\]
for $\tilde{u}^{N}_n:=\tilde{f}^{N}(n)g^{(2)}_n$. The non-vanishing term $c_{1}^{2}$ arises due to the equality 
\[
\left(R_{2}\tilde{u}^{N}\right)_{1}=-c_{1}g_{1}^{(2)}\tilde{f}^{N}(2)-\frac{1}{c_{1}}g_{3}^{(2)}\left(\tilde{f}^{N}(2)-\tilde{f}^{N}(3)\right)=-c_{1},
\]
for all $N\geq2$, while the decaying term is estimated analogously as before. Since $u^{N}\in H_0^2(\N_{0})$, the Rellich inequality with the weight $\rho$ together with identity~\eqref{eq.Rellich.ansatz} imply
\[
 \sum_{n=2}^{\infty}\left(\rho_{n}-\rho^{(2)}_{n}\right)\left|u_{n}^{N}\right|^{2}\leq\|R_{2}\tilde{u}^{N}\|^{2}=c_{1}^{2}+\mathcal O\left(\frac{1}{\log N}\right), \quad N\to\infty.
\]
Taking into account that $u^{N}_{n}\to g_{n}^{(2)}=n^{3/2}$, as $N\to\infty$, for all $n\geq2$, we obtain
\[
\sum_{n=2}^{\infty}n^{3}\!\left(\rho_{n}-\rho_{n}^{(2)}\right)\leq c_{1}^{2}
\]
by monotone convergence. Recalling~\eqref{eq.c1}, we arrive at the statement.
\end{proof}

\begin{remark}
The last theorem implies that the leading term $9/(16n^{4})$ from the asymptotic expansion of the Rellich weight $\rho_{n}^{(2)}$, for $n\to\infty$, is asymptotically optimal.
\end{remark}

\section{A conjecture on the higher order inequalities}\label{sec:conj}

We observe common patterns in our method for the derivation of the improved discrete Hardy and Rellich inequalities. A similar approach could be applicable  to deduce lower bounds for higher integer powers $(-\Delta)^k$ of the discrete Dirichlet Laplacian with $k\geq 3$ on the space $H_0^k(\N_{0})$.
However, the naturally generalized ansatz on the remainder operator $R_{k}$ suggests $R_{k}$ to be a difference operator of order $k$ depending on $k$ unknown coefficient sequences. When compared to $(-\Delta)^k-\rho^{(k)}$, for the weight sequence
\begin{equation}\label{eq.rho.k}
	\rho^{(k)}:= \frac{(-\Delta)^k g^{(k)}}{g^{(k)}}, \quad g^{(k)}_{n} = n^{k-1/2}, \quad n \in \N_0,
\end{equation}
the resulting constraints on the unknown sequences form a rather complicated system of non-linear higher order difference equations, 
whose analysis constitutes a challenging open problem.
Nevertheless, we conjecture that the resulting inequalities hold true.

\begin{conjecture}
 For all $k\in\N$ and $u\in H_0^k(\N_{0})$, we have the inequality
 \begin{equation}
  \sum_{n=k}^{\infty}((-\Delta)^k u)_{n}\overline{u}_{n}\geq\sum_{n=k}^{\infty}\rho^{(k)}_{n}|u_{n}|^{2},
 \label{eq:ineq_higher_order_conj}
 \end{equation}
 where $\rho^{(k)}$ is defined by~\eqref{eq.rho.k}.
\end{conjecture}

Let us take a closer look at the weight 
\[
 \rho_{n}^{(k)}=\frac{\left((-\Delta)^k g^{(k)}\right)_{n}}{g^{(k)}_{n}}=\sum_{j=-k}^{k}(-1)^{j}\binom{2k}{k-j}\frac{g_{n+j}^{(k)}}{g_{n}^{(k)}}.
\]
With the aid of the expansion 
\[
\frac{g_{n+j}^{(k)}}{g_{n}^{(k)}}=\left(1+\frac{j}{n}\right)^{k-1/2}=\sum_{l=0}^{\infty}\binom{k-1/2}{l}\frac{j^{l}}{n^{l}},
\]
for $|j|<n$, a straightforward calculation yields
\begin{equation}
\label{eq.rho.k.asymp}
\rho_{n}^{(k)}=\sum_{l=k}^{\infty}\binom{k-1/2}{2l}\left[\sum_{j=-k}^{k}(-1)^{j}\binom{2k}{k+j}j^{2l}\right]\frac{1}{n^{2l}},
\end{equation}
for $n\geq k\ge1$. By extracting the very first term and simplifying, one obtains the asymptotic formula
\[
 \rho_{n}^{(k)}=\frac{\left((2k)!\right)^{2}}{16^{k}\left(k!\right)^{2}}\frac{1}{n^{2k}}+\mathcal O\!\left(\frac{1}{n^{2k+2}}\right), \quad n\to\infty.
\]
The leading term actually resembles the weight in the continuous higher order Hardy-like inequality for an integer power of the Dirichlet Laplacian on $\R_{+}$, 
which reads 
\begin{equation}
 \int_{0}^{\infty}|u^{(k)}(x)|^2\dd x\geq \frac{\left((2k)!\right)^{2}}{16^{k}\left(k!\right)^{2}} \int_{0}^{\infty}\frac{|u(x)|^{2}}{x^{2k}}\dd x,
\label{eq:higher-order-hardy-cont}
\end{equation}
for $u\in H^{k}(0,\infty)$ with $u(0)=\dots=u^{(k-1)}(0)=0$. In addition, the constant on the right-hand side in~\eqref{eq:higher-order-hardy-cont} is known to be the best possible. It seems that inequality~\eqref{eq:higher-order-hardy-cont} was first established by  M.~{\v S}.~Birman \cite{Birman-1961} with a sketch of the proof. I.~M.~Glazman gives the details of the proof in~\cite[pp.~83--84]{Glazman-1965}. Another proof of~\eqref{eq:higher-order-hardy-cont} was found by M.~P.~Owen in~\cite{Owen-1999}, see also \cite{Gesztesy-Littlejohn-Michael-Wellman_2018}.

We claim that all coefficients in~\eqref{eq.rho.k.asymp} are positive. It is easy to check that
\[
 (-1)^{k}\binom{k-1/2}{2l}>0, \quad \forall l\geq k.
\]
Thus, the coefficients in~\eqref{eq.rho.k.asymp} are positive if and only if
\[
 \sum_{j=-k}^{k}(-1)^{j+k}\frac{j^{2l}}{(k+j)!(k-j)!}>0, \quad \forall l\geq k, 
\]
which is by no means obvious. In fact, in the following lemma, 
we prove an identity from which the positivity readily follows and whose appearance within the context of the Hardy-like inequalities is surprising to the authors. A proof of the identity is worked out in Appendix~\ref{subsec:a.2}.

\begin{lemma}\label{lem:id_amaz}
For all $k,s\in\N_{0}$, one has
\[
\sum_{j=-k}^{k}(-1)^{j+k}\frac{j^{2k+2s}}{(k+j)!(k-j)!}=\sum_{j_1=1}^k \sum_{j_2=1}^{j_1} \cdots \sum_{j_s=1}^{j_{s-1}} (j_1j_2 \cdots j_s)^2.
\]
\end{lemma}

It follows from the positivity of the coefficients in~\eqref{eq.rho.k.asymp} that
\[
\rho^{(k)}_n > \frac{\left((2k)!\right)^{2}}{16^{k}\left(k!\right)^{2}}\frac{1}{n^{2k}},
\]
for all $n\geq k\geq1$. Hence, inequality~\eqref{eq:ineq_higher_order_conj} would improve upon the discrete analogue of~\eqref{eq:higher-order-hardy-cont}.

At the end, we would like to mention a related open problem. The boundary conditions imposed on $u$ by requiring $u\in H_0^1(\N_{0})$ in the Hardy case or $u\in H_0^2(\N_{0})$ in the Rellich case seem natural from the point of view of the corresponding continuous inequalities. On the other hand, there is no actual need for these conditions in the discrete setting. After the identification of $H_0^1(\N_{0})=\{u\in\ell^{2}(\N_{0}) \mid u_{0}=0\}$ with $\ell^{2}(\N)$, the Hardy inequality~\eqref{eq:hardy_ineq_opt} can be rephrased as the operator inequality 
\[
 -\Delta\geq\rho^{(1)}
\]
on the space $\ell^{2}(\N)$, where the discrete Dirichlet Laplacian coincides with the bounded positive operator on $\ell^{2}(\N)$ given by the tri-diagonal Toeplitz matrix
\[
-\Delta=\begin{pmatrix}
2 & -1 \\
-1 & 2 & -1\\
& -1 & 2 & -1\\
& & \ddots & \ddots & \ddots
\end{pmatrix}.
\]

Similarly, one can consider the square of $-\Delta$ on the whole space $\ell^{2}(\N)$ which is a positive bounded operator with the penta-diagonal matrix representation 
\[
(-\Delta)^{2}=\begin{pmatrix}
5 & -4 & 1\\
-4 & 6 & -4 & 1 \\
1 & -4 & 6 & -4 & 1 \\
& 1 & -4 & 6 & -4 & 1 \\
& & \ddots & \ddots & \ddots & \ddots & \ddots
\end{pmatrix}.
\]
The following interesting open problems concern the discrete analogue of the Rellich inequality on $\ell^{2}(\N)$:
\begin{enumerate}
\item Is there $c>0$ such that the inequality   
\[
(-\Delta)^{2}\geq\rho(c),
\]
where $\rho_{n}(c):=c/n^{4}$, for $n\in\N$, holds on $\ell^{2}(\N)$ (in the form sense)?
\item What is the best constant $c$, i.e.~what is 
\[
 \sup\left\{c>0 \mid (-\Delta)^{2}\geq\rho(c)\right\}?
\]
\item Is there an improved or even optimal weight $\rho$ such that $(-\Delta)^{2}\geq\rho$ on $\ell^{2}(\N)$?
\end{enumerate}

%As before, we avoid confusion by omitting the irrelevant first $k-1$ entries of the weight and consider $\rho_k = \{\rho_{k,n}\}_{n\ge k}$. If $k = 2l$ is even with $l \in \N$, inequality~\eqref{eq.conj.form} reads
%\begin{equation}
%	\sum_{n=1}^\infty |((-\Delta)^l u)_n|^2 \ge \sum_{n=k}^\infty \rho_{k,n} |u_n|^2
%\end{equation}
%and if $k = 2l + 1$ is uneven with $l \in \N_0$, it can be written as
%\begin{equation}
%	\sum_{n=1}^\infty |(D(-\Delta)^{l} u)_n|^2 \ge \sum_{n=k}^\infty \rho_{k,n} |u_n|^2.
%\end{equation}

\appendix
\addcontentsline{toc}{section}{Appendices}
\section*{Appendices}

\setcounter{section}{1}
\subsection{Proof of Lemma~\ref{lem.xi.bounds}}\label{subsec:a.1}

For simplicity, we use the following notation 
\[
g_{n}:=g_{n}^{(2)}=n^{3/2} \quad\mbox{ and }\quad h_n := \frac{g_{n+1}}{g_{n}}=\left(1+\frac{1}{n}\right)^{3/2}, \quad n \in \N,
\]
throughout this appendix.
To establish Lemma~\ref{lem.xi.bounds}, we verify that the sequence~$\zeta$ determined recursively by
\begin{equation}
\zeta_{n}=h_{n}\left(4-h_{n+1}-\frac{1}{h_{n-1}}-\frac{h_{n}}{\zeta_{n-1}}\right), \quad n\geq 2,
\label{eq.xi.recur}
\end{equation}
and the initial condition $\zeta_{1} = 8\sqrt{2} - 3\sqrt{3}$, is bounded by the inequalities 
\begin{equation}
 h_{n}h_{n+1} < \zeta_{n} < h_n h_{n+1} h_{n+2},
\label{eq:xi-bounds-upper-lower}
\end{equation}
for all $n\in\N$. 

1) We prove the lower bound from~\eqref{eq:xi-bounds-upper-lower} by induction. 
   The initial hypothesis holds as one can directly compute
\begin{equation}
		\zeta_{1} = 8\sqrt{2} - 3\sqrt{3} \approx 6.1 > 5.2 \approx 3 \sqrt{3} = h_{1}h_{2}.
\end{equation}
For the induction step, suppose that $n \geq 2$ and $\zeta_{n-1} > h_{n-1}h_{n}$. In view of~\eqref{eq.xi.recur}, we then have
\[
 \zeta_{n}>h_{n}\left(4-h_{n+1}-\frac{1}{h_{n-1}}-\frac{h_{n}}{h_{n-1}h_{n}}\right)=h_{n}\left(4-h_{n+1}-\frac{2}{h_{n-1}}\right).
\]
It follows that $\zeta_{n}>h_{n}h_{n+1}$, if
\begin{equation}
		2-h_{n+1}-\frac{1}{h_{n-1}}=2-\left(\frac{n+2}{n+1}\right)^{3/2}-\left(\frac{n-1}{n}\right)^{3/2}>0.
\end{equation}
	
To establish the last inequality, we show that the function 
\begin{equation}
		f(x):= \left(\frac{x+2}{x+1}\right)^{3/2}+\left(\frac{x-1}{x}\right)^{3/2}
\end{equation}
is strictly increasing on the interval $(2,\infty)$. The claimed inequality then follows from that fact that $\lim_{x\to\infty}f(x)=2$ and the continuity of $f$. We readily compute that
	\begin{equation}\label{eq.f.der}
		f'(x)=\frac{3}{2}\left(\sqrt{\frac{x-1}{x^{5}}}-\sqrt{\frac{x+2}{(x+1)^{5}}} \right).
\end{equation}
Thus, $f'(x)>0$, if and only if
\[
 (x-1)(x+1)^{5}-(x+2)x^{5}>0.
\]
The polynomial on the left-hand side can be factorized into three terms over the reals and the resulting inequality reads
\[
 (2x+1)\left(x^{2}+x+\sqrt{2}-1\right)\left(x^{2}+x-\sqrt{2}-1\right)>0.
\]
Each factor is obviously positive for $x>2$.
	
2) We prove the upper bound in~\eqref{eq:xi-bounds-upper-lower}. The proof proceeds again by induction. Since
\begin{equation}
		\zeta_{1}=8\sqrt{2}-3\sqrt{3}\approx 6.1<8=h_{1}h_{2}h_{3},
\end{equation}
it remains to check the induction step. Before starting we note that, for all $n\in\N$, one has
\begin{equation}\label{eq.g.conv}
		g_{n}<\frac{g_{n-1}+g_{n+1}}{2}\quad\mbox{ and }\quad g_{n+1}<\frac{g_{n-1}+g_{n+3}}{2}.
\end{equation}
This follows from the strict convexity of the function $x\mapsto x^{3/2}$ on $[0,\infty)$.	
	
	Suppose that $n\geq 2$ and $\zeta_{n-1}<h_{n-1}h_{n}h_{n+1}$. Similarly as in the proof of the lower bound, using~\eqref{eq.xi.recur} together with the induction hypothesis, we obtain
\[
\zeta_{n}<h_{n}\left(4-h_{n}-\frac{1}{h_{n-1}}-\frac{1}{h_{n-1}h_{n+1}}\right).
\]
To verify the induction step, it suffices to show that
	\begin{equation}
		4-h_{n}-\frac{1}{h_{n-1}}-\frac{1}{h_{n-1}h_{n+1}}<h_{n+1}h_{n+2}.
	\end{equation}
	Recalling that $h_{n}=g_{n+1}/g_{n}$, this is equivalent to the inequality
	\begin{equation}
		4g_{n}g_{n+1}g_{n+2}-g_{n+1}^{2}g_{n+2}-g_{n-1}g_{n+1}g_{n+2}-g_{n-1}g_{n+1}^{2}-g_{n}g_{n+2}g_{n+3}<0.
	\end{equation}
	This, however, can be proved by rewriting and estimating the expression on the left-hand side as follows 
	\begin{equation}
		\begin{aligned}
			& g_{n+1}g_{n+2}(2g_{n}-g_{n+1}-g_{n-1}) + 2g_{n}g_{n+1}g_{n+2}-g_{n-1}g_{n+1}^{2}-g_{n}g_{n+2}g_{n+3} \\
			& < g_{n+1}g_{n+2}(2g_{n}-g_{n+1}-g_{n-1})+ g_{n}g_{n+2}\left(2g_{n+1}-g_{n-1}-g_{n+3}\right) < 0.
		\end{aligned}
	\end{equation}
	Here the first inequality follows from the fact that $g_{n+1}^{2}>g_{n}g_{n+2}$, for all $n \in \N$, and the second from~\eqref{eq.g.conv}.

\subsection{Proof of Lemma \ref{lem:id_amaz}}\label{subsec:a.2}

For $s,k\in\N$, we define three numbers
$$
\begin{aligned}
A_{s,k}&:=\sum_{j_1=1}^k \sum_{j_2=1}^{j_1} \cdots \sum_{j_s=1}^{j_{s-1}} (j_1j_2 \cdots j_s)^2,
\\
B_{s,k}&:=\sum_{j=-k}^{k}\frac{(-1)^{j+k}j^{2k+2s}}{(k+j)!(k-j)!},
\\
C_{s,k}&:=\sum_{m=0}^{2s}(-1)^{m}\binom{2k+2s}{m}S(2k+2s-m,2k)k^{m},
\end{aligned}  
$$
where $S(n,m)$ are the Stirling numbers of the second kind; see~\cite[\S~26.8]{dlmf}. By convention, we also define the three quantities for vanishing indices as follows:
\[
 X_{s,k}:=\begin{cases}
  1, & \quad \mbox{ if }\; s=0, k\geq1, \\
  0, & \quad \mbox{ if }\; s\geq0, k=0,
 \end{cases}
\]
for all $X\in\{A,B,C\}$. Notice that this convention is consistent with the above formulas whenever they are well defined, in particular with $B_{0,k}$ for $k\ge1$, cf.~\eqref{eq.Stirling} below. Lemma~\ref{lem:id_amaz} is nothing but the equality $A_{s,k}=B_{s,k}$. We establish the latter by first showing that $A_{s,k}=C_{s,k}$ and then $B_{s,k}=C_{s,k}$, for all $k,s\in\N_{0}$. For the proof, it suffices to consider $k\ge1$.

The verification of identity $A_{s,k}=C_{s,k}$ is based on generating function formulas.
It is easy to check that, for $k\in\N$, one has
\begin{equation}
 \sum_{s=0}^{\infty}A_{s,k}t^{s}=\prod_{j=1}^{k}\frac{1}{1-j^{2}t}
\label{eq:gener_func_A}
\end{equation}
for $|t|<1/k^{2}$. In fact, $A_{s,k}$ can be identified with the complete homogeneous symmetric polynomial
\[
 h_{s}(x_{1},\dots,x_{k}):=\sum_{1\leq i_{1}\leq\dots\leq i_{s}\leq k}x_{i_1}\dots x_{i_s},
\]
for $x_{i}=i^{2}$. Then~\eqref{eq:gener_func_A} follows from the well known generating function formula
\[
 \sum_{s=0}^{\infty}h_{s}(x_{1},\dots,x_{k})t^{s}=\prod_{j=1}^{k}\frac{1}{1-x_{j}t},
\]
see~\cite[Chap.~7]{Stanley-1999}. Indeed, the above series converges and is equal to the product on the right hand side if $|x_jt|<1$ for all $1\le j\le k$. Formula~\eqref{eq:gener_func_A} implies
\begin{equation}
 \sum_{s=0}^{\infty}A_{s,k}t^{2s+2k}=\prod_{j=1}^{k}\frac{t^{2}}{1-j^{2}t^{2}},
\label{eq:gener_func_A_modif}
\end{equation}
for all $k\in\N$ and $|t|<1/k$.
We prove the same formula holds for $C_{s,k}$, too. As a~consequence, $A_{s,k}=C_{s,k}$ for all $s,k\in\N_{0}$.

\begin{lemma}
For $k\in\N$ and $|t|<1/k$, one has
\[
 \sum_{s=0}^{\infty}C_{s,k}t^{2s+2k}=\prod_{j=1}^{k}\frac{t^{2}}{1-j^{2}t^{2}}.
\]
\end{lemma}

\begin{proof}
It follows from the exponential generating function formula~\cite[Eq.~26.8.12]{dlmf}
\[
\frac{(e^{t}-1)^{k}}{k!}=\sum_{j=0}^{\infty}S(j,k)\frac{t^{j}}{j!},
\]
where the series converges for $k\in\N$ and $t\in\C$, that 
\[
\frac{1}{(2k)!}\left(\frac{e^{t}-1}{t}\right)^{2k}=\sum_{j=0}^{\infty}S(2k+j,2k)\frac{t^{j}}{(2k+j)!},
\]
for all $k\in\N$ and $t\in\C$. Multiplying this function by 
\[
 e^{-kt}=\sum_{m=0}^{\infty}\frac{(-kt)^{m}}{m!},
\]
one obtains
\[
\frac{1}{(2k)!}\left(\frac{e^{t/2}-e^{-t/2}}{t}\right)^{2k}=\sum_{s=0}^{\infty}\left(\sum_{m=0}^{s}(-1)^{m}\frac{S(2k+s-m,2k)}{(2k+s-m)!m!}k^{m}\right)t^{s}.
\]
Taking also into account that the function on the left-hand side is even for all $k\in\N$, we deduce the formula
\begin{equation}
\frac{1}{(2k)!}\left(e^{t/2}-e^{-t/2}\right)^{2k}=\sum_{s=0}^{\infty}C_{s,k}\frac{t^{2s+2k}}{(2s+2k)!},
\label{eq:C_exp_gen_func}
\end{equation}
for all $k\in\N$ and $t\in\C$.

It is clear that the function $f(t):=e^{t/2}-e^{-t/2}$ fulfills $f^{(l)}(0)\geq0$ for all $l\in\N_{0}$. 
By the Leibniz rule the same holds true for any non-negative integer power of $f$. Thus, we see from~\eqref{eq:C_exp_gen_func} that $C_{s,k}\geq0$ for all $s\in\N_{0}$ and $k\in\N$. Next, we multiply both sides of~\eqref{eq:C_exp_gen_func} by $\exp(-pt)$, where $p>0$, 
and integrate with respect to~$t$ from $0$ to $\infty$. The non-negativity of $C_{s,k}$ justifies the interchange of the sum and the integral. We arrive at the equation
\[
\sum_{s=0}^{\infty}C_{s,k}p^{-2s-2k-1}=\frac{1}{(2k)!}\int_{0}^{\infty}e^{-pt}\left(e^{t/2}-e^{-t/2}\right)^{2k}\dd t.
\]
Finally, with the aid of the Laplace transform identity~\cite[Eq.~2.3.1.2]{Prudnikov-etal-vol4}
\[
 \int_{0}^{\infty}e^{-pt}\sinh^{2k}(at)\dd t=\frac{(2k)!}{p}\prod_{j=1}^{k}\frac{a^{2}}{p^{2}-(2aj)^{2}},
\]
which holds true if $p>2ka$, we deduce that 
\[
\sum_{s=0}^{\infty}C_{s,k}p^{-2s-2k}=\prod_{j=1}^{k}\frac{1}{p^{2}-j^{2}},
\]
for $p>k$. The statement readily follows.
\end{proof}

It remains to prove the identity $B_{s,k}=C_{s,k}$ for $s\ge0$ and $k\ge1$. This can be done directly as follows. Using the binomial formula, we get
\begin{align*}
 B_{s,k}&=\sum_{j=0}^{2k}\frac{(-1)^{j}}{(2k-j)!j!}(k-j)^{2k+2s}\\
 &=\sum_{j=0}^{2k}\frac{(-1)^{j}}{(2k-j)!j!}\sum_{m=0}^{2k+2s}\binom{2k+2s}{m}k^{2k+2s-m}(-j)^{m}\\
 &=\sum_{m=0}^{2k+2s}(-1)^{m}\binom{2k+2s}{m}k^{2k+2s-m}\sum_{j=0}^{2k}(-1)^{j}\frac{j^{m}}{(2k-j)!j!},
\end{align*}
for $k\in\N$ and $s\in\N_{0}$. Applying the identity~\cite[Eq.~26.8.6]{dlmf}
\begin{equation}\label{eq.Stirling}
 S(n,k)=\sum_{j=0}^{k}(-1)^{k+j}\frac{j^{n}}{(k-j)!j!}
\end{equation}
and recalling the convention $S(n,k)=0$, if $n<k$, we obtain
\begin{align*}
 B_{s,k}&=\sum_{m=2k}^{2k+2s}(-1)^{m}\binom{2k+2s}{m}S(m,2k)k^{2k+2s-m}\\
 &=\sum_{m=0}^{2s}(-1)^{m}\binom{2k+2s}{m}S(2k+2s-m,2k)k^{m}.
\end{align*}
This implies that $B_{s,k}=C_{s,k}$, for all $s,k\in\N_{0}$.

\subsection*{Acknowledgment}
The authors acknowledge the support of the EXPRO grant No.~20-17749X
of the Czech Science Foundation.

{\footnotesize
	\bibliographystyle{acm}
	%\bibliography{C:/Users/Borbala/Nextcloud/01_Borbala/meta/ref}
	\bibliography{ref_09}
}

\end{document}